\documentclass[12pt]{amsart}
\usepackage{txfonts}
\usepackage{amscd}
\usepackage{cite}
\usepackage{epsfig}
\usepackage{verbatim}
\usepackage{mathdots}
\usepackage{amssymb}
\usepackage{amsfonts}
\usepackage{amsbsy}
 \usepackage{graphicx}
 \usepackage[usenames]{color} 
 \usepackage{cancel}

\newtheorem{theo}{Theorem}[section]

\newtheorem{prop}[theo]{Proposition}
\newtheorem{coro}[theo]{Corollary}

\theoremstyle{definition}

\numberwithin{equation}{section}

\newcommand{\R}{{\mathbb R}}
\newcommand{\Z}{{\mathbb Z}}

\newcommand{\var}{{\vartheta}}

\begin{document}

\title{Spectral measures with arbitrary Hausdorff dimensions}

\author{Xin-Rong Dai}
\address{School of Mathematics and Computational Science, Sun Yat-sen University, Guangzhou, 510275,   China}
\email{daixr@mail.sysu.edu.cn}

\author{Qiyu Sun}
\address{Department of Mathematics, University of Central Florida, Orlando, FL 32816, USA}
\email{qiyu.sun@ucf.edu}

\thanks{The research  is partially  supported by the National Science Foundation of China (No. 10871180 and 11371383),
 and the National Science Foundation (DMS-1109063). }


\date{\today}
\maketitle

\begin{abstract}
In this paper, we consider spectral properties of Riesz product measures supported on homogeneous Cantor sets
 and we show the existence of spectral measures with arbitrary Hausdorff dimensions, including non-atomic zero-dimensional spectral measures
 and one-dimensional singular  spectral measures.
\end{abstract}


\bigskip

\section{\bf Introduction}
\setcounter{equation}{0}

Given sequences  ${\mathcal B}:=\{b_n\}_{n=1}^\infty$
and ${\mathcal D}:=\{d_n\}_{n=1}^\infty$  of positive integers that satisfy
\begin{equation}\label{bndn.eq}
1<d_n<b_n,\ \ n=1, 2, \cdots,
\end{equation}
we let
\begin{equation}
\label{rhon.def}
\rho_1:=1 \quad {\rm and} \quad \rho_n:=\prod_{j=1}^{n-1} b_j \  \ {\rm for} \ \  n\ge 2, \end{equation}
and we define
\begin{equation}\label{H-Cantor}
C(\mathcal{B},\mathcal{D}):= 
 \sum_{n=1}^\infty  \frac{\Z/d_n\cap [0, 1)}{\rho_n}.
\end{equation}
The set $C({\mathcal B}, {\mathcal D})$  
is a {\em homogeneous Cantor set} contained in the interval  $[0, \sum_{n=1}^\infty (d_n-1)(d_n\rho_n)^{-1}]$.
The reader may refer to
\cite{Fal, PS,FWW} 
on  homogenous Cantor sets. 

Define the Fourier transform  $\hat \mu$ of a probability measure $\mu$
by $\hat \mu(\xi):= \int_{\R} e^{-2\pi i\xi x}d\mu(x)$.
In this paper, we  consider
 the {\em Riesz product  measure} $\mu_{\mathcal{B}, \mathcal{D}}$
defined by
\begin{equation} \label {Rieszproduct.def}
\widehat{\mu_{\mathcal{D}, \mathcal{B}}}(\xi):=\prod_{n=1}^\infty H_{d_n}\Big(\frac{\xi}{d_n\rho_{n}}\Big),
\end{equation}
where
$$H_m(\xi):=\frac1{m} \sum_{j=0}^{m-1} e^{-2\pi i j \xi }= \frac{1-e^{-2\pi   mi \xi}}{m(1-e^{-2\pi i \xi})}, \ \ m\ge 1.$$
The  Riesz product measure $\mu_{{\mathcal B}, {\mathcal D}}$
is supported on the homogeneous Cantor set
$C({\mathcal B}, {\mathcal D})$ \cite{Fal, FWW},
and it becomes
 the  {\em Cantor measure} $\mu_{b, d}$  when $b_n=b$ and $d_n=d$ for all $n\ge 1$ \cite{D1, D2, DHL, DHS}.

A probability measure $\mu$ with compact
support is said to be a {\em spectral measure} if there exists a countable set $\Lambda$ of real numbers, called  a {\it spectrum}, such that $\{e^{-2\pi i \lambda x}: \lambda \in \Lambda \}$ forms an orthonormal basis for $L^2(\mu)$.
A classical example of spectral measures is the Lebesgue measure on $[0, 1]$, for which the set of integers is a spectrum.
Spectral properties for a probability measure are one of fundamental problems in Fourier
analysis and they have close connection to tiling as formulated in Fuglede's spectral set conjecture
\cite{F, T1, T2, T4, T3,  LW2, LW3, SW}.
In 1998,   Jorgensen and Pedersen \cite{JP} discovered the first families of non-atomic singular spectral measures, particularly Cantor measures $\mu_{b, 2}$
  with $4\le b\in 2\Z$. Since then, various  singular spectral measures on self-similar/self-affine fractal sets
  have  been found, see for instance \cite{D1, D2, DHL, DHLau, DHS, 
    DL, HLL, HuL, T1, JP, LW, LW2, Lai,   Li1, Li2,  S, W}.
  In this paper, we  consider
spectral properties of   Riesz product measures $\mu_{{\mathcal B}, {\mathcal D}}$  supported on {\bf non-self-similar}
homogenous Cantor sets $C({\mathcal B}, {\mathcal D})$.

\begin{theo} \label{main1}
Let  ${\mathcal B}:=\{b_n\}_{n=1}^\infty$
and ${\mathcal D}:=\{d_n\}_{n=1}^\infty$ be sequences of positive integers
that satisfy \eqref{bndn.eq} and
\begin{equation}\label{bndn.eq2}
2\le  {b_n}/{d_n}\in \Z \quad {\rm for \ all} \ \ n\ge 1.
\end{equation}
Then
\begin{equation} \label{spectrum}
\Lambda_{{\mathcal B}, {\mathcal D}}:=\bigcup_{L=1}^\infty  \Big( \sum_{n=1}^L \big([0, d_n)\cap \Z\big) \rho_n\Big)
\end{equation}
 is a spectrum of the Riesz product measure $\mu_{\mathcal{B}, \mathcal{D}}$
 in \eqref{Rieszproduct.def}.
\end{theo}

For a probability measure $\mu$,  define its Hausdorff dimension $\dim_H(\mu)$ by
\begin{equation*}
\dim_H(\mu):=\inf\left\{\dim_H(E): \ \mu(E)=1\right\},
\end{equation*}
where $\dim_H(E)$ is the Hausdorff dimension of a set $E$.
It is known that  Cantor measures $\mu_{b, 2}$ with $4\le b\in 2\Z$
have their Hausdorff dimension $\ln 2/\ln b$.
Next we estimate  Hausdorff dimension of the Riesz product measure $\mu_{{\mathcal B}, {\mathcal D}}$, 
with its proof given in the appendix.

 \begin{prop}\label{rieszdimension.prop}
Let $0\le \alpha\le 1$, and let ${\mathcal B}$
and ${\mathcal D}$ be sequences   of positive integers
that satisfy \eqref{bndn.eq}, 
\begin{equation}\label{dnlimit}
\lim_{n\to \infty} d_n=+\infty,
\end{equation}
and
\begin{equation}\label{measuredimension}
\lim_{n\to \infty} \frac{\ln d_n}{\ln b_n}=\alpha.
\end{equation}
Then the Riesz product measure $\mu_{{\mathcal B}, {\mathcal D}}$ in  \eqref{Rieszproduct.def} has Hausdorff dimension $\alpha$,
\begin{equation*} 
\dim_H(\mu_{{\mathcal B}, {\mathcal D}})=\alpha.\end{equation*}
\end{prop}

Our main contribution of this paper,  the existence of
 spectral  measures with {\bf arbitrary} Hausdorff dimension in $[0,1]$, follows immediately from
 Theorem \ref{main1} and Proposition \ref{rieszdimension.prop}.

\begin{coro} \label{main2}
Let $0\le \alpha\le 1$, and let ${\mathcal B}$
and ${\mathcal D}$ be sequences   of positive integers
that satisfy \eqref{bndn.eq}, \eqref{bndn.eq2}, \eqref{dnlimit} and \eqref{measuredimension}.
Then  the Riesz product measure
 $\mu_{\mathcal{B}, \mathcal{D}}$ in \eqref{Rieszproduct.def} is a spectral measure with Hausdorff dimension $\alpha$.
\end{coro}

Taking $\alpha=0$ and $1$ in Corollary \ref{main2} leads to the existence of zero-dimensional non-atomic spectral measures
 and one-dimensional singular spectral measures respectively, cf.
 Fuglede's conjecture that any spectral set with positive Lebesgue measure is a tile.

\section{Maximal orthogonal sets of Riesz product measures}

A spectral measure may admit various spectra.
Fourier series corresponding to different spectra could have completely different convergence rate
\cite{S, DHS13}. To study spectra of a probability measure $\mu$, we recall a weak notation,
 \emph{maximal orthogonal set}  $\Lambda$, which means that
$\{e^{-2\pi i \lambda x}: \lambda \in \Lambda \}$ is a maximal orthogonal set of $L^2(\mu)$.
 As $\Lambda$ is a maximal orthogonal set (spectrum) of a probability measure $\mu$ if and only if its shift $\Lambda+t$ is for any real $t\in \R$.
  So in this paper we may normalize maximal orthogonal sets (spectra) by assuming that they contain the origin.
In 2009,  Dutkay, Han and Sun  made their first attempt to characterize
 maximal orthogonal sets of fractal measures in \cite{DHS},
  where  a maximal orthogonal set of the one-fourth Cantor measure $\mu_{4,2}$
 is labeled as a binary tree with each vertex having finite regular lengths, see
    \cite{DHL} and references therein for general Cantor measures $\mu_{b,d}$ with $2\le b/d\in \Z$.
    In this section, we
    first consider labeling a maximal orthogonal set of the Riesz product $\mu_{{\mathcal B}, {\mathcal D}}$ on homogeneous Cantor set
    $C({\mathcal B}, {\mathcal D})$. 

For labeling a maximal orthogonal set, we introduce some notation.
Let $\Sigma_d :=\{0,1,\ldots, d-1 \}$ for $d\ge 1$.
For a sequence $\mathcal{D}:=\{d_n\}_{n=1}^\infty$ of positive integers, let
 $\Sigma_\mathcal{D}^0:=\vartheta$,
$\Sigma_\mathcal{D}^n:=\Sigma_{d_1} \times \Sigma_{d_2} \times \cdots \times \Sigma_{d_n}$ for $n\ge 1$,
and $\Sigma_\mathcal{D}^\ast:=\cup_{n=0}^\infty \Sigma_\mathcal{D}^n$ be the set of all finite words.
 We say that
 a tree is
 a {\it ${\mathcal D}$-adic tree} if
it  has $\var$,
 $\Sigma_\mathcal{D}^n$  and $\{\delta i,\  i\in \Sigma_{d_{n+1}}\}$
  as its root, the set of all $n$-th level nodes,
 and the set of offsprings of $\delta\in \Sigma_\mathcal{D}^n, n\ge 1$, respectively,
 where  $\var \delta :=\delta$ and
$\delta\delta'$ is the concatenation of words
$\delta \in \Sigma_{d_1} \times \Sigma_{d_2} \times \cdots \times \Sigma_{d_n}$ 
and
$\delta'\in \Sigma_{d_{n+1}} \times \Sigma_{d_{n+2}} \times \cdots \times \Sigma_{d_{n+m}}$. 
   Given sequences $\mathcal{B}=\{b_n\}_{n=1}^\infty$ and $\mathcal{D}=\{d_n\}_{n=1}^\infty$  of positive integers satisfying
\eqref{bndn.eq},
we say that
$\tau: \Sigma_\mathcal{D}^\ast\rightarrow \R$
is  a {\em maximal tree mapping} if
\begin{itemize}
\item[{(i)}] \ $\tau(\vartheta) =\tau(R_n(0^\infty)) =0$  for all $n\geq 1$;
\item[{(ii)}] \ $\tau (\delta_1\cdots \delta_n) \in (\delta_n+d_n\Z)\cap \{-\lfloor b_n/2\rfloor, -\lfloor b_n/2\rfloor+1, \ldots,
b_n-1-\lfloor b_n/2\rfloor \}$ for $\delta_1\cdots \delta_n\in \Sigma_{\mathcal D}^n, n\ge 1$;
and
\item[{(iii)}]  for any word $\delta\in \Sigma_\mathcal{D}^{n}$ 
 there exists
$\delta'\in \Sigma_{d_{n+1}} \times \Sigma_{d_{n+2}} \times \cdots \times \Sigma_{d_{n+m}}
$ of length $m\ge 1$ such that $\tau(R_k(\delta \delta' 0^\infty))=0$ for sufficiently large $k$,
\end{itemize}
where ${0}^{\infty} := 000\cdots$ and $R_k(\delta):=\delta_1\cdots \delta_k\in \Sigma_{\mathcal D}^k$ for $\delta=\delta_1\cdots\delta_{k}\delta_{k+1}\cdots
\in \otimes_{n=1}^\infty \Sigma_{d_n}$.
For a maximal tree mapping $\tau$, define
\begin{equation}\label{lambdatau.def}
\Lambda(\tau):=\Big\{ \sum_{n=1}^\infty \tau( R_n(\delta 0^\infty))\rho_{n}:\ \ \delta\in \Sigma_{\mathcal D}^\ast
\ {\rm with}\
\tau( R_n(\delta 0^\infty))=0 \ {\rm for \ sufficiently \ large} \ n \Big\},
\end{equation}
where
  $\rho_n, n\ge 1$, are given in \eqref{rhon.def}.
Following the argument used in \cite{DHL, D2}, we can characterize  maximal orthogonal sets of the  Riesz product measure $\mu_{\mathcal{B},\mathcal{D}}$
in \eqref{Rieszproduct.def}
by  maximal tree mappings.

\begin{theo}\label{th1.6}
Let  sequences ${\mathcal B}$
and ${\mathcal D}$  of positive integers
 satisfy \eqref{bndn.eq} and \eqref{bndn.eq2}, $\mu_{\mathcal{B},\mathcal{D}}$ be the Riesz product measure
in \eqref{Rieszproduct.def}, and for a maximal tree mapping $\tau$ let $\Lambda(\tau)$ be the set given in  \eqref{lambdatau.def}. Then $\Lambda$ is a maximal orthogonal set of
the Riesz product measure   $\mu_{\mathcal{B},\mathcal{D}}$ that contains the origin
 if and only if $\Lambda=\Lambda(\tau)$ for some maximal tree mapping $\tau$. 
\end{theo}

Denote by $\#(E)$ the cardinality of a finite set $E$, and define the
 {\em upper Beurling dimension} $\dim^+(\Lambda)$ of a discrete set $\Lambda$ of real numbers by
$$\dim^+(\Lambda):=\inf \left\{r>0: \  \limsup_{h\to \infty}
\sup_{x\in \R} \frac{\# (\Lambda\cap [x-h, x+h])}{(2h)^{r}}<\infty\right\}.$$
 Given  sequences ${\mathcal B}$
and ${\mathcal D}$  satisfying \eqref{bndn.eq},
\eqref{bndn.eq2}, \eqref{dnlimit} and \eqref{measuredimension}, one may verify that
the set $\Lambda(\tau)$ associated with a maximal tree mapping $\tau$ has
upper Beurling dimension 
 being less than or equal to
 Hausdorff dimension of the homogeneous Cantor set $C({\mathcal B}, {\mathcal D})$,
\begin{equation}\label{beurlingdimension.req}
\dim^+(\Lambda(\tau))\le \dim_H (C({\mathcal B}, {\mathcal D})).\end{equation}
The above result is established in \cite{DHSW}
for maximal orthogonal sets of  Cantor measures $\mu_{b, d}$
with $2\le b/d\in \Z$.
We remark that unlike Fourier frames on the unit interval \cite{Lan},
 spectra of Cantor measures with zero upper Beurling dimension  has been constructed by Dai, He and Lai in \cite{DHL}.
  The reader may refer to \cite{D1, D2, DHL, DHLau, DHS, DHS13, DHSW, 
  DL, HLL, HuL, T1, JP, LW, LW2, Lai, Lan, Li1, Li2, PS,
  S, W}  and  references therein
  for additional information on
self-similar/self-affine spectral measures.

By \eqref{beurlingdimension.req} and Theorem \ref{th1.6}, a necessary condition for a  countable set
 to be a spectrum of  the Riesz product measure $\mu_{{\mathcal B}, {\mathcal D}}$ is that
its upper Beurling dimension is less than or equal to the Hausdorff dimension of the measure $\mu_{{\mathcal B}, {\mathcal D}}$.
The above necessary condition is far from being sufficient. In fact, it is a very challenging problem to find
appropriate sufficient conditions, see \cite{D2, DHS, DHL} and references therein for recent advances.
In this paper, we provide a simple sufficient condition for spectra of
Riesz product measures.

\begin{theo}\label{th1.7}
Let  ${\mathcal B}:=\{b_j\}_{j=1}^\infty$
and ${\mathcal D}:=\{d_j\}_{j=1}^\infty$ be sequences of positive integers
that satisfy \eqref{bndn.eq} and \eqref{bndn.eq2},
$\tau$ be a maximal tree mapping and let $\Lambda=\Lambda(\tau)$ be as in
\eqref{lambdatau.def}.
Assume that
\begin{equation}\label{th1.7.necessary-1}
 \#\big\{n\ge 1, \tau( R_{n}(\delta 0^\infty))\ne 0\big\}<\infty \quad {\rm for \ all} \ \ \delta\in \Sigma_{\mathcal D}^\ast
\end{equation}
and
\begin{equation}\label{th1.7.necessary}
\sup_{n\ge 1}
\sup_{\delta\in \Sigma_{\mathcal D}^n} \sum_{j=1}^\infty \left(\frac{|\tau( R_{n+j}(\delta 0^\infty))|}{b_{n+j}}\right)^2<\infty,
\end{equation}
then  $\Lambda$ is a spectrum of the Riesz product measure  $\mu_{{\mathcal B}, {\mathcal D}}$ in \eqref{Rieszproduct.def}.
\end{theo}

As an application, we have the following immediately:

\begin{coro}\label{coro1.7}
Let  ${\mathcal B}, {\mathcal D}, \tau$ and $\Lambda(\tau)$ be as in
Theorem \ref{th1.7}. If
\begin{equation} \label{coro1.7.necessary}
\sup_{n\ge 1}
\sup_{\delta\in \Sigma_{\mathcal D}^n} \#\big\{j\ge 1, \tau( R_{n+j}(\delta 0^\infty))\ne 0\big\}<\infty,
\end{equation}
then  $\Lambda(\tau)$ is a spectrum of the Riesz product measure  $\mu_{{\mathcal B}, {\mathcal D}}$ in \eqref{Rieszproduct.def}.
\end{coro}

The requirements \eqref{th1.7.necessary-1} and
\eqref{th1.7.necessary} are clearly weaker than the one in \eqref{coro1.7.necessary},
since
$$d_{n+j}\le |\tau( R_{n+j}(\delta 0^\infty))|\le b_{n+j}/2\quad {\rm for \ all} \ \delta\in \Sigma_{\mathcal D}^n\ {\rm and} \ j\ge 1.$$
We remark that
those requirements are not equivalent in general
when ${\mathcal B}$ and  ${\mathcal D}$ satisfy \eqref{dnlimit} and
\eqref{measuredimension} for some $0\le \alpha<1$.

\smallskip

For   sequences $\mathcal{B}$ and $\mathcal{D}$  satisfying
\eqref{bndn.eq} and
\eqref{bndn.eq2}, one may verify that
 the map 
  defined by
$$\tau_{{\mathcal B}, {\mathcal D}} (\delta_1\cdots \delta_n):=\delta_n\ \  {\rm for} \ \ \delta_1\ldots \delta_n\in \Sigma_{\mathcal D}^n\ \ {\rm and}\ \ n\ge 0,$$
is a maximal tree mapping satisfying \eqref{coro1.7.necessary}, and that
the corresponding set $\Lambda(\tau_{{\mathcal B}, {\mathcal D}})$ is same as the spectral set $\Lambda_{{\mathcal B}, {\mathcal D}}$ in \eqref{spectrum},
\begin{equation*}\label{lambdaspectrum}
\Lambda(\tau_{{\mathcal B}, {\mathcal D}})=\Lambda_{{\mathcal B}, {\mathcal D}}.
\end{equation*}
Therefore the spectral conclusion in Theorem \ref{main1} follows from Corollary \ref{coro1.7}. 

\bigskip
\section {Spectra of  Riesz product measures}

In this section, we prove Theorems  \ref{th1.7}.
 For that purpose, we recall a characterization about
spectra of a probability measure $\mu$  with compact support, given by Jorgensen and Pederson in \cite{JP}, which states that
{\em  a countable set
$\Lambda$ containing zero is a spectrum for $L^2(\mu)$ if and only if
\begin{equation}\label{Q.def}
Q(\xi):=\sum_{\lambda \in \Lambda} |\hat{\mu}(\xi+\lambda)|^2 \equiv 1
 \ \ {\rm for \ all}\ \   \xi \in \R.
\end{equation}
}
Denote by $\deg(G)$ the degree of a trigonometric polynomial $G$.
Recall that $Q(\xi)$ in \eqref{Q.def} is a real analytic function. Then the proof of Theorem \ref{th1.7} reduces to establishing the following
general theorem.

\begin{theo}\label{main1.general.tm}
Let sequences ${\mathcal B}:=\{b_n\}_{n=1}^\infty$ and ${\mathcal D}:=\{d_n\}_{n=1}^\infty$
 of positive integers
 satisfy \eqref{bndn.eq} and \eqref{bndn.eq2},
$\tau$ be a maximal tree mapping satisfying \eqref{th1.7.necessary-1} and \eqref{th1.7.necessary}, and let $\Lambda(\tau)$ be as in
\eqref{lambdatau.def}.
Assume that  $\{G_n\}_{n=1}^\infty$ is a family of trigonometric polynomials satisfying
$G_n(0)=1$,
\begin{equation} \label{main1.general.tm.eq2}
\sum_{l=0}^{d_n-1} |G_n(\xi+l/d_n)|^2=1, \ \ \xi\in \R,
\end{equation}
\begin{equation} \label{main1.general.tm.eq1}
{\rm deg}(G_n)\le D_0 d_n,
\end{equation}
and
 \begin{equation}\label{main1.general.tm.eq3}
 \inf_{d_n\xi\in [-2/3, 1/2]} |G_n(\xi)|\ge D_1,
\end{equation}
where $D_0, D_1$ are positive constants independent of $n\ge 1$.
%
%
%
%
Define a compactly supported distribution $\phi$ with help of its Fourier transform by
\begin{equation} \label{main1.general.tm.eq4}
\hat \phi(\xi):=\prod_{n=1}^\infty G_n\Big(\frac{\xi}{d_n\rho_n}\Big),
\end{equation}
where $\{\rho_n\}_{n=1}^\infty$ is given in \eqref{rhon.def}.
Then
\begin{equation} \label{main1.general.tm.eq5}
\sum_{\lambda\in \Lambda (\tau)} |\hat \phi(\xi+\lambda)|^2=1
\quad {\rm for \ all} \ \ \xi \in [0, 1/2].
\end{equation}
\end{theo}

\begin{proof} Observe from \eqref{main1.general.tm.eq2} that
\begin{equation}\label{main1.general.tm.pf.eq9}
\|G_n\|_\infty:=\sup_{\xi\in \R} |G_n(\xi)|\le 1 \end{equation}
for all $n\ge 1$.
By  \eqref{main1.general.tm.eq1}, \eqref{main1.general.tm.pf.eq9} and Bernstein inequality for trigonometric polynomials,
we obtain that
\begin{equation} \label{main1.general.tm.pf.eq1-}
|G_n(\eta/d_n)-1|\le   \|G_n^\prime\|_\infty |\eta/d_n|\le D_0 \|G_n\|_\infty |\eta|\le D_0|\eta|
\end{equation}
and
\begin{eqnarray} \label{main1.general.tm.pf.eq1}
0 & \le  & 1-|G_n(\eta/d_n)|\le
1- |G_n(\eta/d_n)|^2\nonumber\\
& \le &
\big(\|G_n^{\prime\prime}\|_\infty\|G_n\|_\infty+\|G_n^\prime\|_\infty^2\big)\
 |\eta/d_n|^2 
\le  2D_0^2 |\eta|^2
\end{eqnarray}
for all  $\eta\in [-1, 1]$ and $n\ge 1$.
 Thus
 $$\sum_{n=1}^\infty  \Big|G_n\Big(\frac{\xi}{d_n\rho_n}\Big)-1\Big|
 \le  D_0 \sum_{n=1}^\infty \rho_n^{-1}\le D_0\sum_{n=1}^\infty 4^{1-n}= \frac{4D_0}{3}
 $$
 by  \eqref{bndn.eq}, \eqref{bndn.eq2} and \eqref{main1.general.tm.pf.eq1-}.
Therefore  the infinite product in  \eqref{main1.general.tm.eq4}
 is well-defined and $\phi$ is a  compactly supported  distribution.

\smallskip

Let $C_0>0$ be so chosen that
\begin{equation}\label{C0.def}
\max(D_1, 1- 2D_0^2t^2)\ge \exp(-C_0 t^2)\quad {\rm for \ all} \ \ 0\le t\le 2/3.
\end{equation}
Then for $L\ge 1$ and $\eta\in [-2\rho_L/3, \rho_L/2]$,
we  obtain from
\eqref{bndn.eq}, \eqref{bndn.eq2},
\eqref{main1.general.tm.eq3}, \eqref{main1.general.tm.pf.eq1} and \eqref{C0.def}
that
\begin{equation} \label{main1.general.tm.pf.eq5}
 \Big|G_L\Big( \frac{\eta}{d_L\rho_L}\Big)\Big|
 \ge \max\big(D_1, 1- 2D_0^2(|\eta|/\rho_L)^2\big)\ge \exp\big(-C_0\big({|\eta|}/{\rho_L}\big)^2\big) 
\end{equation}
 and
\begin{eqnarray} \label{main1.general.tm.pf.eq7}
\prod_{n=L}^{\infty} \Big|G_n\Big( \frac{\eta}{d_n\rho_n}\Big)\Big|
&\ge &  \prod_{n=L}^{\infty} \exp\big(-C_0\big({|\eta|}/{\rho_n}\big)^2\big)\nonumber\\
& \ge &  \prod_{n=L}^{\infty} \exp\big(-C_0(|\eta|/\rho_L)^2 \times 4^{2(n-L)}\big) \ge
\exp\big(-2C_0(|\eta|/\rho_L)^2\big).
\end{eqnarray}

\smallskip

For $\xi\in [0, 1/2]$ and $\delta\in \Sigma_{\mathcal D}^L$,
we obtained from \eqref{bndn.eq}, \eqref{rhon.def}, \eqref{bndn.eq2} and the definition of a maximal tree mapping that
\begin{equation}\label{main1.general.tm.pf.eq3}
\xi+\sum_{k=1}^L \tau(R_k(\delta 0^\infty))\rho_k\ge -\sum_{k=1}^L \lfloor b_k/2\rfloor \rho_k\ge  -\sum_{k=1}^L \frac{\rho_{k+1}}{2} \ge -\frac{2}{3}
\rho_{L+1}
\end{equation}
and
\begin{equation} \label{main1.general.tm.pf.eq4}
\xi+\sum_{k=1}^L \tau(R_k(\delta 0^\infty))\rho_k\le  \frac{1}{2}+\sum_{k=1}^L (b_k-1-\lfloor b_k/2\rfloor) \rho_k
\le \frac{1}{2}+\sum_{k=1}^L \frac{\rho_{k+1}-\rho_k}{2}\le \frac{1}{2}\rho_{L+1}.
\end{equation}

For $\delta\in \Sigma_{\mathcal D}^L$, let
$${\mathcal K}(\delta)=\{k\ge 1,\  R_{L+k}(\delta 0^\infty)\ne 0\}.$$
For the nontrivial case that  ${\mathcal K}(\delta)\ne \emptyset$,
there exist finitely many positive integers
$n_1<n_2<\ldots<n_K$
by
 \eqref{th1.7.necessary-1}
 such that
$${\mathcal K}(\delta)=\{n_1, n_2, \ldots, n_K\}.$$
Set $n_{K+1}=+\infty$ and $\lambda=\sum_{k=1}^\infty \tau(R_k(\delta 0^\infty))\rho_k$.
Write
\begin{eqnarray*} 
\Big|\prod_{n=1}^L G_n\Big(\frac{\xi+\lambda}{d_n\rho_n}\Big)\Big|
& = & |\hat \phi(\xi+\lambda)|
\times
\left( \prod_{n=L+1}^{L+n_{1}-1}
\Big| G_n\Big(\frac{\xi+\sum_{k=1}^{L} \tau(R_k(\delta 0^\infty))\rho_k}{d_n\rho_n}\Big)\Big|\right)^{-1}\nonumber\\
& &
\quad \times
\Big|G_{L+n_1}\Big(\frac{\xi+\sum_{k=1}^{L} \tau(R_k(\delta 0^\infty))\rho_k}{d_{L+n_1}\rho_{L+n_1}}\Big)\Big|^{-1}\nonumber\\
& & \quad
\times \left(\prod_{i=2}^{K} \Big|G_{L+n_i}\Big(\frac{\xi+\sum_{k=1}^{L+n_{i-1}} \tau(R_k(\delta 0^\infty))\rho_k}{d_{L+n_i}\rho_{L+n_i}}\Big)\Big|
\right)^{-1}\nonumber
\\
& & \quad
\times \left(\prod_{i=1}^K
 \left( \prod_{n=L+n_i+1}^{L+n_{i+1}-1}
\Big| G_n\Big(\frac{\xi+\sum_{k=1}^{L+n_i} \tau(R_k(\delta 0^\infty))\rho_k}{d_n\rho_n}\Big)\Big|\right)^{-1}\right)\nonumber\\
\end{eqnarray*}
Then by \eqref{main1.general.tm.pf.eq9}, 
 \eqref{main1.general.tm.pf.eq5}--\eqref{main1.general.tm.pf.eq4} and the  definition of a maximal tree mapping,
 we get
\begin{eqnarray}\label{main1.general.tm.pf.eq10}
& & \Big|\prod_{n=1}^L G_n\Big(\frac{\xi+\lambda}{d_n\rho_n}\Big)\Big|\nonumber\\
& \le &  \exp\left(C_0 \left(\frac{2\rho_{L+1}}{3\rho_{L+n_1}}\right)^2
+ C_0\sum_{i=2}^K \left(\frac{ (|\tau(R_{L+n_{i-1}}(\delta 0^\infty))|+2/3) \rho_{L+n_{i-1}}}{\rho_{L+n_i}}\right)^2 \right)
\nonumber\\
& & \quad
\times \left(\prod_{i=1}^K
 \left( \prod_{n=L+n_i+1}^{\infty}
\Big| G_n\Big(\frac{\xi+\sum_{k=1}^{L+n_i} \tau(R_k(\delta 0^\infty))\rho_k}{d_n\rho_n}\Big)\Big|\right)^{-1}\right)\nonumber\\
& & \quad
\times
\left( \prod_{n=L+1}^{\infty}
\Big| G_n\Big(\frac{\xi+\sum_{k=1}^{L} \tau(R_k(\delta 0^\infty))\rho_k}{d_n\rho_n}\Big)\Big|^{-1}\right)
\times |\hat \phi(\xi+\lambda)|\nonumber\\
& \le &
 \exp\left(C_0 +4 C_0\sum_{i=2}^K \left(\frac{ |\tau(R_{L+n_{i-1}}(\delta 0^\infty))|}{b_{L+n_{i-1}}}\right)^2 \right)\nonumber\\
 & &  \times
 \exp\left(2C_0 + 2C_0
 \sum_{i=2}^K \left(\frac{ (|\tau(R_{L+n_{i}}(\delta 0^\infty)|+2/3) \rho_{L+n_{i}}}{\rho_{L+n_i+1}}\right)^2
\right)
\times |\hat \phi(\xi+\lambda)|
\nonumber\\
\qquad &\le &  \exp\left(3C_0+ 12 C_0\sum_{j=1}^\infty\left(\frac{ |\tau(R_{j+L}(\delta 0^\infty)| }{b_{j+L}}\right)^2\right)\times
 |\hat \phi(\xi+\lambda)|, \ \  \xi\in [0, 1/2].
 \end{eqnarray}

For the trivial case that ${\mathcal K}(\delta)=\emptyset$,
$$\lambda:=\sum_{k=1}^\infty \tau(R_k(\delta 0^\infty))\rho_k=\sum_{k=1}^L \tau(R_k(\delta 0^\infty))\rho_k$$
and for $\xi\in [0, 1/2]$,
\begin{eqnarray} \label{main1.general.tm.pf.eq8}
\Big|\prod_{n=1}^L G_n\Big(\frac{\xi+\lambda}{d_n\rho_n}\Big)\Big|
 & = &
 \left(\prod_{n=L+1}^\infty
\Big| G_n\Big(\frac{\xi+\sum_{k=1}^L \tau(R_k(\delta 0^\infty))\rho_k}{d_n\rho_n}\Big)\Big|\right)^{-1} \times |\hat \phi(\xi+\lambda)|
\nonumber\\
& \le &   \exp(8 C_0/9) |\hat \phi(\xi+\lambda)|,
\end{eqnarray}
where the last inequality follows from  \eqref{main1.general.tm.pf.eq7}, \eqref{main1.general.tm.pf.eq3}
and \eqref{main1.general.tm.pf.eq4}.

Define
$$\Lambda_L=\Big\{\sum_{k=1}^\infty \tau(R_{k}(\delta 0^\infty)) \rho_k: \ \ \delta\in \Sigma_{\mathcal D}^L\Big\}, \ L\ge 1.$$
The  sets $\Lambda_L, L\ge 1$, are well-defined and  satisfy
\begin{equation}\label{main1.general.tm.pf.eq11-}
\Lambda_1\subset \Lambda_2\subset \cdots \subset \Lambda_L\to \Lambda(\tau) \ \  {\rm as}\  \ L\to +\infty\end{equation}
by \eqref{th1.7.necessary-1}.
Combining \eqref{th1.7.necessary},  \eqref{main1.general.tm.pf.eq10}
and \eqref{main1.general.tm.pf.eq8} leads to  the existence of an absolute constant $C$ such that
\begin{equation}\label{main1.general.tm.pf.eq11}
\Big|\prod_{n=1}^L G_n\Big(\frac{\xi+\lambda}{d_n\rho_n}\Big)\Big|
\le C|\hat \phi(\xi+\lambda)|
\end{equation}
for all $\xi\in [0, 1/2)$ and $\lambda\in \Lambda_L$.

By \eqref{rhon.def}, \eqref{bndn.eq2}, \eqref{main1.general.tm.eq2} and the definition of a maximal  tree mapping,
we can prove
\begin{eqnarray}\label{main1.general.tm.pf.eq12}
& &  \sum_{\lambda\in \Lambda_L}
\Big|\prod_{n=1}^L G_n\Big(\frac{\xi+\lambda}{d_n\rho_n}\Big)\Big|^2 =
 \sum_{\delta\in \Sigma_{\mathcal D}^L}
\Big|\prod_{n=1}^L G_n\Big(\frac{\xi+\sum_{k=1}^\infty \tau(R_{k}(\delta 0^\infty)) \rho_k}{d_n\rho_n}\Big)\Big|^2
\nonumber\\
\qquad & = &  \sum_{\delta_1\in \Sigma_{d_1-1}} \cdots \sum_{\delta_L\in \Sigma_{d_L-1}}
\prod_{n=1}^L \Big| G_n\Big( \frac{\xi+\sum_{k=1}^{n-1} \tau(\delta_1\delta_2\cdots\delta_k)\rho_k +\delta_n\rho_n
}{d_n\rho_n}\Big)\Big|^2\nonumber\\
\qquad & = &  \sum_{\delta_1\in \Sigma_{d_1-1}} \cdots \sum_{\delta_{L-1}\in \Sigma_{d_{L-1}-1}}
\prod_{n=1}^{L-1} \Big| G_n\Big( \frac{\xi+\sum_{k=1}^{n-1} \tau(\delta_1\delta_2\cdots\delta_k)\rho_k +\delta_n\rho_n
}{d_n\rho_n}\Big)\Big|^2\nonumber\\
& & \quad \times
\left(\sum_{\delta_L\in \Sigma_{d_L-1}}
\Big| G_L\Big( \frac{\xi+\sum_{k=1}^{L-1} \tau(\delta_1\delta_2\cdots\delta_k)\rho_k}{d_L\rho_L} +\frac{\delta_L}{d_L}\Big)\Big|^2\right)
\nonumber\\
& = & \sum_{\delta_1\in \Sigma_{d_1-1}} \cdots \sum_{\delta_{L-1}\in \Sigma_{d_{L-1}-1}}
\prod_{n=1}^{L-1} \Big| G_n\Big( \frac{\xi+\sum_{k=1}^{n-1} \tau(\delta_1\delta_2\cdots\delta_k)\rho_k +\delta_n\rho_n
}{d_n\rho_n}\Big)\Big|^2\nonumber\\
& = & \cdots=1
\end{eqnarray}
by induction on $L\ge 1$.

By \eqref{main1.general.tm.eq4}, \eqref{main1.general.tm.pf.eq9} and \eqref{main1.general.tm.pf.eq12}, we conclude that
$$\sum_{\lambda\in \Lambda_L}|\hat \phi(\xi+\lambda)|^2\le
\sum_{\lambda\in \Lambda_L}
\Big|\prod_{n=1}^L G_n\Big(\frac{\xi+\lambda}{d_n\rho_n}\Big)\Big|^2=1.$$
Then taking limit $L\to \infty$ in the above inequality and using \eqref{main1.general.tm.pf.eq11-} yield
\begin{equation}\label{main1.general.tm.pf.eq13}
\sum_{\lambda\in \Lambda(\tau)} |\hat\phi(\xi+\lambda)|^2\le 1.
\end{equation}

Given an arbitrary $\epsilon>0$ and $L\ge 1$, there exist an integer $M\ge L$
by \eqref{bndn.eq}, \eqref{bndn.eq2}, \eqref{main1.general.tm.eq4} and  \eqref{main1.general.tm.pf.eq1-} such that
\begin{equation}\label {main1.general.tm.pf.eq14}
\Big|\prod_{n=1}^{M} G_n\Big(\frac{\xi+\lambda}{d_n\rho_n}\Big)\Big|\le (1+\epsilon) |\hat\phi(\xi+\lambda)|
\end{equation}
for all $\xi\in [0, 1/2)$ and $\lambda\in \Lambda_L$.
By \eqref{main1.general.tm.pf.eq11}, \eqref{main1.general.tm.pf.eq12} and \eqref{main1.general.tm.pf.eq14}, we obtain that
\begin{eqnarray*}
1 & = & \sum_{\lambda\in \Lambda_{M}}
\Big|\prod_{n=1}^{M} G_n\Big(\frac{\xi+\lambda}{d_n\rho_n}\Big)\Big|^2\\
& = & \left(\sum_{\lambda\in \Lambda_{L}}+ \sum_{\lambda\in \Lambda_{M}\backslash \Lambda_L}\right)
\Big|\prod_{n=1}^{M} G_n\Big(\frac{\xi+\lambda}{d_n\rho_n}\Big)\Big|^2\\
& \le & (1+\epsilon) \sum_{\lambda\in \Lambda_L} |\hat \phi(\xi+\lambda)|^2+
C \sum_{\lambda\in \Lambda_{M}\backslash \Lambda_L} |\hat \phi(\xi+\lambda)|^2\\
&\le & (1+\epsilon) \sum_{\lambda\in \Lambda_L} |\hat \phi(\xi+\lambda)|^2+ C
\sum_{\lambda\in \Lambda(\tau)\backslash \Lambda_L} |\hat \phi(\xi+\lambda)|^2.
\end{eqnarray*}
Taking limit $L\to +\infty$ and using \eqref{main1.general.tm.pf.eq11-} and \eqref{main1.general.tm.pf.eq13}, we have that
\begin{equation*}
1\le (1+\epsilon) \sum_{\lambda\in \Lambda(\tau)} |\hat \phi(\xi+\lambda)|^2\le 1+\epsilon.
\end{equation*}
This completes the proof of the desired equation \eqref{main1.general.tm.eq5}
as $\epsilon>0$ is  chosen arbitrarily.
\end{proof}

We remark that trigonometric polynomials satisfying  \eqref{main1.general.tm.eq2},
known as {\it multi-channel quadrature mirror filters}, are important for the construction of multiband
orthonormal wavelets \cite{BDS, Vai}. The requirement
\eqref{main1.general.tm.eq3} could be thought as a  weak version of Mallat's condition for
a  scaling function to have orthonormal shifts, cf.
\cite{CohenS, Daubechies, Mallat}.

Trigonometric polynomials $G_n(\xi), n\ge 1$, in Theorem \ref{main1.general.tm}
have factors  $H_{d_n}(\xi)=\sum_{j=0}^{d_n-1} e^{-2\pi i j \xi }/d_n$ by
\eqref{main1.general.tm.eq2}.
If  $G_n(\xi), n\ge 1$, are further  assumed to have factors $(H_{d_n}(\xi))^N$ for some $N\ge 2$ \cite{BDS, Vai},
then one may establish the conclusion in Theorem \ref{main1.general.tm}
 with the requirement \eqref{th1.7.necessary}
 replaced by the following weaker  assumption:
$$\sup_{n\ge 1} \sup_{\delta\in \Sigma_{\mathcal D}^n}
\sum_{j=1}^\infty \left( \frac{|\tau(R_{n+j}(\delta 0^\infty))|}{b_{n+j}}\right)^{2N}<\infty.$$

\begin{appendix}
\bigskip
\section {\bf Homogeneous Cantor sets}

In this appendix, we evaluate Hausdorff dimensions of homogeneous Cantor sets
and prove Proposition \ref{rieszdimension.prop}.

Given sequences $\mathcal{D}:=\{d_n\}_{n=1}^\infty$ and   $\mathcal{R}:=\{r_n\}_{n=1}^\infty$ of positive  numbers
that satisfy
$ 2\le d_n\in \Z$ and $r_n d_n\le 1$ for all $n\ge 1$,
define the {\em homogeneous Cantor set}  $\mathcal{E}(\mathcal{R}, \mathcal{D})$ by
\begin{equation}\label{HCantor.new} \mathcal{E}(\mathcal{R}, \mathcal{D}):=\cap_{n=0}^\infty\cup_{\delta\in\Sigma_{\mathcal{D}}^n} J_\delta,
\end{equation}
 where  $\{J_\delta: \delta\in\Sigma_{\mathcal{D}}^\ast\}$ is the family of closed intervals contained in $J_\vartheta:=[0,1]$  such that
for each $\delta \in \Sigma_{\mathcal D}^n, n\ge 0$, subintervals
$J_{\delta k}, k\in \Sigma_{d_{n+1}}$,
 of $J_{\delta}$ satisfy the following:
(i)\  $J_{\delta k}$  has same length $r_{n+1} |J_{\delta}|$ for every $k\in \Sigma_{d_{n+1}}$;
(ii)\ the gaps between $J_{\delta k}$ and $J_{\delta (k+1)}$ have same length for all $0\le k<d_{n+1}-2$; and (iii)\
 the left endpoint of $J_{\delta 0}$ is the same as the left endpoint of $J_{\delta}$, and the right
 endpoint of $J_{\delta (d_{n+1}-1)}$ is the same as the right endpoint of $J_{\delta}$
\cite{Fal, FWW, PS}.
The above homogeneous Cantor set $\mathcal{E}(\mathcal{R}, \mathcal{D})$
has its Hausdorff dimension
\begin{equation}\label{Hausdorff-D}
\dim_H (\mathcal{E}(\mathcal{R},\mathcal{D}))= \liminf_{n\rightarrow\infty} \frac {\sum_{j=1}^n\ln{q_j}}{\sum_{j=1}^n \ln{1/r_j}},
\end{equation}
see for instance \cite{FWW}.

The set $C(\mathcal{B},\mathcal{D})$ in \eqref{H-Cantor} can be obtained from rescaling
the homogeneous Cantor set $ \mathcal{E}(\mathcal{R}, \mathcal{D})$  in
\eqref{HCantor.new}. In particular,
\begin{equation}
C(\mathcal{B},\mathcal{D})= \Big (\sum_{n=1}^\infty\frac {d_n-1}{d_n\rho_{n}}\Big)\ \mathcal{E}(\mathcal{R},\mathcal{D})\end{equation}
 with  $\mathcal{R}=\{r_n\}_{n=1}^\infty$ given by
$$
r_n=\frac{\sum_{j=n+1}^\infty (d_j-1)/(d_j \rho_{j})} {\sum_{j=n}^\infty (d_j-1)/(d_j \rho_{j})}, \ n\ge 1.
$$
For sequences ${\mathcal B}$ and ${\mathcal D}$ of positive integers satisfying
\eqref{dnlimit} and
\eqref{measuredimension},
$$ \Big| \rho_n \sum_{j=n}^\infty \frac{d_j-1}{d_j \rho_{j}}-1\Big|\le
\frac{1}{d_n}+\sum_{j=n+1}^\infty \left (\max_{i\ge n} \frac{1}{b_i}\right)^{j-n}\to 0\ \ {\rm as} \ \  n\to \infty,
$$
 which implies that
$\lim_{n\to \infty}r_nb_n=1$.
 Combining the above limit with \eqref{measuredimension} and \eqref{Hausdorff-D} leads to
 \begin{equation}\label{dimensionhcs}
\dim_H (C(\mathcal{B},\mathcal{D}))= \lim_{n\to \infty} \frac {\sum_{j=1}^n\ln{d_j}}{\sum_{j=1}^n\ln{1/r_j}}=\lim_{n\rightarrow\infty} \frac {\ln{d_n}}{\ln{ b_n}}=\alpha.
\end{equation}
Recall that
the Riesz product  measure
$\mu_{{\mathcal B}, {\mathcal D}}$  in \eqref{Rieszproduct.def} is  the  natural measure  on
$C(\mathcal{B},\mathcal{D})$.
Then
the Hausdorff dimension of the Riesz product measure $\mu_{{\mathcal B}, {\mathcal D}}$
in Proposition \ref{rieszdimension.prop}
follows from \eqref{dimensionhcs}.
\end{appendix}



\begin{thebibliography}{99}
\smallskip



\bibitem{BDS} {\sc  N. Bi, X.-R. Dai and Q. Sun},
{\it Construction of compactly supported $M$-band wavelets}, Appl. Comp. Harmonic Anal., 6(1999), 113-131.

\bibitem{CohenS} {\sc A. Cohen and Q. Sun},
{\it  An arithmetic characterization of the conjugate quadrature filters associated to orthonormal wavelet bases},
 SIAM J. Math. Anal., 24(1993), 1355-1360.


\bibitem{D1}
{\sc X.-R. Dai}, {\it When does a Bernoulli convolution admit a
spectrum?},  Adv. Math., 231(2012),  1681-1693.


\bibitem{D2} {\sc X.-R. Dai},  {\it Spectra and maximal orthogonal sets of Cantor measures},  arXiv:1401.4630

\bibitem {DHL}
{\sc X.-R. Dai, X.-G. He and C.-K. Lai,} {\it Spectral property of
Cantor measures with consecutive digits},  Adv. Math., 242(2013),  187-208.

\bibitem {DHLau} {\sc X.-R. Dai, X.-G. He and K.-S. Lau,} {\it On spectral $N$-Bernoulli measures}, Adv. Math., 259(2014),  511-531.  

\bibitem{Daubechies}
I. Daubechies,  {\it Ten Lectures on Wavelets}, SIAM, 1992.

\bibitem{DHS13} {\sc D. Dutkay, D. Han and Q. Sun},
 {\it Divergence of mock and scrambled Fourier series on fractal measures}, Trans. Amer. Math. Soc.,  366(2014), 2191-2208.

\bibitem {DHS}
{\sc D. Dutkay, D. Han and Q. Sun}, {\it On spectra of a Cantor
measure}, Adv. Math., 221(2009), 251-276.


\bibitem {DHSW}
{\sc D. Dutkay, D. Han, Q. Sun and E. Weber}, {\it On the Beurling
dimension of exponential frames}, Adv. Math., 226(2011), 285-297.




\bibitem {DL}
{\sc D. Dutkay and C.-K. Lai}, {\it Uniformity of measures with
Fourier frames}, Adv. Math.,  252(2014), 684-707.


\bibitem {Fal}
{\sc K. J. Falconer}, {\it Fractal Geometry, Mathematical Foundations
and Applications}, Wiley, New York, 1990.


\bibitem {FWW}
{\sc D.-J. Feng, Z.-Y. Wen and J. Wu}, {\it Some dimensional results for homogeneous Moran sets}, Sci. China, Series A, 40(1997), 475-482.


\bibitem {F}
{\sc B. Fuglede}, {\it Commuting self-adjoint partial differential
operators and a group theoretic problem}, J. Funct. Anal., 16(1974), 101-121.


\bibitem  {HLL}
{\sc X.-G. He, C.-K. Lai and K.-S. Lau}, {\it Exponential spectra in
$L^2(\mu)$}, Appl. Comput. Harmon. Anal.,
 34(2013), 327-338.


\bibitem  {HuL}
{\sc T.-Y. Hu and K.-S. Lau}, {\it Spectral property of the
Bernoulli convolutions}, Adv. Math., 219(2008), 554-567.

\bibitem  {T1}
{\sc A. Iosevich, N. Katz and T. Tao}, {\it Convex bodies with a point of curvature do not admit exponential bases}, Amer. J. Math., 123(2001), 115-120.

\bibitem  {T2}
{\sc A. Iosevich, N. Katz and T. Tao}, {\it Fuglede conjecture holds for convex planar domains}, Math. Res. Lett., 10(2003), 559-569.


\bibitem  {JP}
{\sc P. Jorgensen and S. Pedersen}, {\it Dense analytic subspaces in
fractal $L^2$ spaces}, J. Anal. Math., 75(1998), 185-228.

\bibitem {T4}
{\sc M. Kolountzakis and M. Matolcsi}, {\it  Tiles with no spectra},
Forum Math., 18(2006), 519-528.

\bibitem  {LW}
{\sc I. {\L}aba and Y. Wang}, {\it On spectral Cantor measures}, J.
Funct. Anal., 193(2002), 409-420.


\bibitem  {LW2}
{\sc J. C. Lagarias  and Y. Wang}, {\it Self-affine tiles in $\R^n$}, Adv. Math., 121(1996), 21-49.


\bibitem  {LW3}
{\sc J. C. Lagarias  and Y. Wang}, {\it Tiling the line by the translates of one tile}, Invent. Math., 124(1996), 341-365.


\bibitem {Lai}
{\sc C.-K. Lai}, {\it On Fourier frame of absolutely continuous
measures}, J. Funct. Anal., 261(2011), 2877-2889.

\bibitem  {Lan}
{\sc H. Landau}, {\it Necessary density conditions for sampling and
interpolation of certain entire functions}, Acta Math., 117(1967),
37-52.



\bibitem {Li1}
{\sc J.-L. Li}, {\it $\mu_{M,D}-$orthogonality and compatible pair},
J. Funct. Anal., 244(2007), 628-638.

\bibitem {Li2}
{\sc J.-L. Li}, {\it Spectra of a class of self-affine measures}, J.
Funct. Anal., 260(2011),  1086-1095.

\bibitem{Mallat} S. Mallat, {\it Multiresolution approximations and wavelet orthonormal bases of $L^2(\R)$}, Trans. Amer. Math. Soc., 315(1989), 69-87.



\bibitem {PS}
{\sc Y. Peres and B. Solomyak}, {\it Self-similar measures and intersection of Cantor sets}, Trans.
Amer. Math. Soc., 350(1998),  4065-4087.





\bibitem  {S}
{\sc R. S. Strichartz}, {\it Convergence of mock Fourier series}, J. Anal.
Math., 99(2006), 333-353.


\bibitem  {SW}
{\sc R. S. Strichartz and Y. Wang}, {\it Geometry of self-affine tiles I}, Indiana Univ. Math. J., 48(1999), 1-23.


\bibitem  {T3}
{\sc T. Tao}, {\it Fuglede's conjecture is false in 5 or higher dimensions}, Math. Res. Lett., 11(2004), 251-258.

\bibitem{Vai} {\sc P. P. Vaidynathan}, {\it Multirate Systems and Filter Banks}, Prentice-Hall, 1993.


\bibitem  {W}
{\sc Y. Wang}, {\it Wavelets, tiling and spectral sets}, Duke Math. J., 114(2002), 43-57.



\end{thebibliography}
\end{document}